\documentclass{amsart}
\usepackage{amssymb,latexsym}
\usepackage{amsmath}
\usepackage[dvipdfm]{graphicx}
\usepackage{amscd}
\usepackage{color}
\usepackage{enumerate}
\newenvironment{enumeratei}{\begin{enumerate}[\upshape (i)]}{\end{enumerate}}
\numberwithin{equation}{section}
\theoremstyle{plain}
 \newtheorem{theorem}{Theorem}[section]
 \newtheorem{lemma}[theorem]{Lemma}

\theoremstyle{definition}
 \newtheorem{definition}[theorem]{Definition}
 
 \newtheorem{notationdef}[theorem]{Notation and definition}
 \newtheorem{example}[theorem]{Example}
 
 \newtheorem{conclude}[theorem]{Concluding remark}

\newcommand \subne[1] {\textup{Sub}_{01}(#1)}
\newcommand \catb {\mathbf{POS}_{01}^{{0\textup{s}}}} 
\newcommand \sst [1] {\scriptscriptstyle #1}
\newcommand \lcat [1] {\textup{Cat}(#1)}
\newcommand \subecat[1] {\textup{CatSub}^{\sst{\textup{em}}}_{01}(#1) }
\newcommand \catprsub[1] {\textup{CatPrSub}_{01}(#1) }
\newcommand \fprinc {\textup{Princ}}
\newcommand \princ[1] {\fprinc(#1)}
\newcommand\hatnu {{\hat\nu}}
\newcommand\hatTheta {{\hat\Theta}}

\newcommand \dd[1] {{\displaystyle{#1}}}
\newcommand \zetai[2] {\zeta_{#1,#2}}

\newcommand \cg[2] {\textup{con}(#1,#2)}
\newcommand \congen[2] {\textup{con}_{#1}(#2)}
\newcommand \cgi[3] {\textup{con}_{#1}(#2,#3)}

\newcommand\alg [1] {{\mathcal #1}}
\newcommand \pairs [1] {{\textup{Pairs}^{\leq}(#1)}}
\newcommand \Quord[1]{\textup{Quord}(#1)}

\newcommand \quos[1] {\textup{quo}(#1)}
\newcommand \iquos[2] {\textup{quo}_{#1}(#2)}
\newcommand \biquos[2] {\textup{quo}_{#1}\Bigl(#2\Bigr)}
\newcommand \blokk[2] {[#1]#2}%Block of #1 by #2
\newcommand \restrict[2] {{#1\rceil_{#2}}}

\newcommand \tuple [1] {\langle #1\rangle}
\newcommand \pair [2] {\tuple{#1,#2}}
\renewcommand\phi{\varphi}
\renewcommand\emptyset{\varnothing}
\newcommand\red[1]{{\textcolor{red}{#1}}}

\newcommand \tbf [1] {\textbf{#1}} 
\newcommand \set[1] {\{#1\}}

\renewcommand\phi{\varphi}
\renewcommand\epsilon{\varepsilon}

\newcommand \zigzags {\mathcal Z}
\newcommand \vesz[1]{{#1^{\scriptscriptstyle{{{\mathord{ \blacktriangleright  }}}}}}}
\newcommand\init [1] {} %Initials of names
\newcommand\nothing [1] {}
\begin{document}
\title[Families of monotone maps by principal lattice congruences]
{Representing some families of monotone maps by  principal lattice congruences}

\author[G.\ Cz\'edli]{G\'abor Cz\'edli}
\email{czedli@math.u-szeged.hu}
\urladdr{http://www.math.u-szeged.hu/\textasciitilde czedli/}
\address{University of Szeged, Bolyai Institute, 
Szeged, Aradi v\'ertan\'uk tere 1, HUNGARY 6720}

%% Thanks
\thanks{This research was supported by
NFSR of Hungary (OTKA), grant number 
K83219}

%% Dedication 
%\dedicatory{To the memory of Andr\'as P.\ Huhn}

\subjclass {06B10\red{.\hfill September 4, 2014}}
%
%
%\subjclass[2010] 06B10 (1980-now) Ideals, congruence relations 


\begin{abstract} For a lattice $L$ with 0 and 1, let $\princ L$ denote the ordered set of principal congruences of $L$.   For $\{0,1\}$-sublattices $A\subseteq B$ of $L$, congruence generation defines a natural map  $\princ A\to\princ B$. In this way, we obtain a small category of bounded ordered sets as objects and some 0-separating $\set{0,1}$-preserving monotone maps as morphisms such that every hom-set consists of at most one morphism.  We prove the converse:  each small category of bounded ordered set  with these properties  is representable by principal lattice congruences in the  above sense. 
\end{abstract}

\maketitle
\section{Introduction}
Motivated by the history of congruence lattice representation problem, which  reached its summit in \init{F.\ }Wehrung~\cite{wehrung} and  \init{P.\ }R\r{u}\v{z}i\v{c}ka~\cite{ruzicka},
 \init{G.\ }Gr\"atzer~\cite{ggprincl} has recently started an analogous new topic of lattice theory.  Namely, for a lattice $L$, let $\princ L=\pair{\princ L}{\subseteq}$ denote the   ordered set  of principal congruences of $L$. A congruence is \emph{principal} if it is generated by a pair $\pair ab$ of elements.
Ordered sets (also called  partially ordered sets or posets) and lattices with 0 and 1 are called \emph{bounded}. If $L$ is a bounded lattice, then $\princ L$ is a bounded ordered set.  Conversely,  by \init{G.\ }Gr\"atzer~\cite{ggprincl}, each  bounded  ordered set $P$ is isomorphic to $\princ L$ for an appropriate bounded lattice $L$.
 The ordered sets $\princ L$ of countable lattices $L$ were characterized by \init{G.\ }Cz\'edli~\cite{czgprincc}. 
In this paper, we give a simultaneous representation for  many bounded ordered sets together with some monotone maps by principal lattice congruences.

\section{Our result}  Given two bounded ordered sets, $P$ and $Q$, a map $\psi\colon P\to Q$ is called a \emph{$\set{0,1}$-preserving monotone map} if $\psi(0_P)=0_Q$, $\psi(1_P)=1_Q$, and, for all $x,y\in P$,  $x\leq_P y$ implies $\psi(x)\leq_Q \psi(y)$. If, in addition, $0_P$ is the only preimage of $0_Q$, that is, if $\psi^{-1}(0_Q)=\set{0_P}$, then we say that $\psi$ is a \emph{$0$-separating $\set{0,1}$-preserving monotone map}.
The category of bounded ordered sets with 0-separating bound-preserving monotone maps will be denoted by $\catb$. 
For a lattice $L$ and $x,y\in L$, the principal congruence generated by $\pair xy$ is denoted by $\cg xy$ or $\cgi Lxy$. If $L$ is bounded, $K$ is a sublattice of $L$, and $0_L, 1_L\in K$, then  $K$ is a $\set{0,1}$-\emph{sublattice} of $L$. In this case, the map
\begin{equation}\zetai KL\colon\princ K\to \princ L,\quad
\text{defined by}\quad \cgi Kxy\mapsto \cgi Lxy
\label{lsddZ}
\end{equation}
is clearly a 0-separating $\set{0,1}$-preserving monotone map. It is easy to see that $\zetai KL(\cgi Kxy)$ is the same as $\congen L{\cgi Kxy}$.
We shall soon see that each 0-separating $\set{0,1}$-preserving monotone map is, in a good sense, of the form \eqref{lsddZ}.

\begin{definition}\label{sldzTg} Let  $\psi$ be a morphism in  $\catb$, that is, let $\psi$ be 
0-separating $\set{0,1}$-preserving monotone map from 
a bounded ordered set
$\tuple{P_0;\leq_0}$ to another one, $\tuple{P_1;\leq_1}$. We say that $\psi$ is \emph{representable by principal lattice congruences} if  there exist a bounded lattice $L_1$,  a $\set{0,1}$-sublattice $L_0$ of $L_1$,  and order isomorphisms 
\[\text{$\xi_0\colon \tuple{P_0;\leq_0} \to\tuple{\princ {L_0};\subseteq }$\quad and\quad $\xi_1\colon \tuple{P_1;\leq_1} \to\tuple{\princ {L_1};\subseteq }$}
\]
such that $\psi=  \xi_1^{-1}\circ \zetai{L_0}{L_1}\circ \xi_0$; that is, the diagram
\begin{equation}
\begin{CD}
\tuple{P_0;\leq_0}  @>{\dd\psi}>> \tuple{P_1;\leq_1 }\\
@V{\dd{\xi_0}}VV   @A{\dd{\xi_1^{-1}}}AA   \\
\tuple{\princ {L_0};\subseteq } @>{\dd{\zetai {L_0}{L_1}}}>>  \tuple{\princ{ L_1};\subseteq }
\end{CD}\\ \label{uuvkdhHnf}
\end{equation}
is commutative. (We compose maps from right to left, that is, $(\psi_0\circ \psi_1)(x)= \psi_0(\psi_1(x))$.) 
\end{definition}

We know from  Cz\'edli~\cite{czgmonprinc}  that each morphism of $\catb$ is representable by principal lattice congruences.  Now, we are going to deal with many morphisms of $\catb$ simultaneously; this is motivated by the simple observation that  $L_1$ above has many $\set{0,1}$-sublattices $L_i$ in general, and  we obtain a lot of 
0-separating $\set{0,1}$-preserving monotone maps of the form $\zetai {L_i}{L_j}$. 
To formulate exactly what we intend to do, the most economic  way is to use the rudiments of category theory.

\begin{notationdef}\ 
\begin{enumeratei}
\item
As usual, we often consider an ordered set $S=\tuple{S;\leq}$ as a small category. This category, denoted by  $\lcat S$, consists of the elements of $S$ as objects, and pairs belonging to the ordering relation $\leq$ as morphisms. 
\item
For a bounded lattice $L$, let $\subne L$ denote the lattice of all $\set{0,1}$-sublattices of $L$, ordered by set inclusion. Also, 
let $\subecat L$ denote the category whose objects are the $\set{0,1}$-sublattices of $L$ and the morphism are the inclusion maps. That is, for $J,K\in \subne L$, we have a unique morphism $\iota_{J,K}\colon   J\to K$, acting identically on $J$, if $J\subseteq K$, and we have no  $J\to K$ morphism otherwise.  
\item For $S$ and $L$ as above, a functor $E\colon \lcat S\to \subecat L$ is an \emph{embedding functor} if it is an order-embedding in the usual sense. That is, if $E$ is a functor and, for $s,t\in S$, $E(s)\subseteq E(t)$ if{}f $s\leq t$. 
\item  Associated with $L$,  we also have to consider the  category, $\catprsub L$, whose set of objects is 
$\set{ \princ K: K\in\subne L }$, and whose morphisms are the $\zeta_{J,K} \colon \princ J\to \princ K$, given in \eqref{lsddZ}. Note that there is at most one $J\to K$ morphism; there is one if{}f $J\subseteq K$. 
\item In this way, $\fprinc$ becomes a functor from $\subecat L$ to  $\catprsub L$ in the natural way.  For $J\subset K$, the $\fprinc$ image of this inclusion will be denoted by $\zeta_{J,K}$ rather than $\princ{J\subseteq K}$.
\end{enumeratei} 
\end{notationdef}

\begin{definition}
Let $S$ be a bounded ordered set and let $F\colon \lcat S\to \catb$ be a functor.  We say that $F$ is \emph{representable by principal lattice congruences} if there exist  a bounded lattice $L$
such that $F$ is naturally isomorphic (also called naturally equivalent) to the functor $\fprinc\circ E$.
\end{definition}

Our main result is the following.

\begin{theorem}\label{thmmain} For every  bounded ordered set $S$,  every functor \[F\colon \lcat S\to \catb\] is  representable by principal lattice congruences. 
\end{theorem}

\begin{example}To show the strength of this theorem and  to help in understanding it,  now we derive the result of Cz\'edli~\cite{czgmonprinc}, mentioned above. Let $\psi\colon\tuple{P_0;\leq_0}\to\tuple{P_1;\leq_1}$ be a morphism in $\catb$. We want to show that $\psi$ is representable by principal congruences in the sense of Definition~\ref{sldzTg}.  Let $S=\set{0,1}$ be the 2-element lattice, and define a functor $F$ by $F(i)=\tuple{P_i;\leq_i}$, for $i\in S=\set{0,1}$, and $F(0<1)=\psi$. By Theorem~\ref{thmmain}, we have a natural isomorphism $\xi \colon F\to \fprinc \circ E$. This means that, for $i\in S$, $\xi(i)\colon F(i)\to 
(\fprinc \circ E)(i)=\fprinc(E(i))$ is an isomorphism in $\catb$, and that the diagram
\begin{equation}
\begin{CD}
F(0)=\tuple{P_0;\leq_0}  @>{\dd\psi}>> F(1)= \tuple{P_1;\leq_1 }\\
@V{\dd{\xi(0)}}VV   @V{\dd{\xi_1}}VV   \\
(\fprinc\circ E)(0)=  \tuple{\princ {L_0};\subseteq } @>{\dd{\zetai {L_0}{L_1}}}>>  (\fprinc\circ E)(1)=\tuple{\princ {L_1};\subseteq }
\end{CD}\\ \label{vkuuHnf}
\end{equation}
commutes. Here $L_i$ stands for $E(i)$. Comparing  \eqref{uuvkdhHnf} and \eqref{vkuuHnf}, we obtain that $\psi$ is representable by principal congruences.
\end{example}

\section{Lemmas and proofs}
Due to  the heavy machinery developed in \init{G.\ }Cz\'edli~\cite{czgprincc},  the proof of Theorem~\ref{thmmain}  here is  short.

A \emph{quasiordered set} is a structure $\tuple{H;\nu}$ where $H\neq \emptyset$ is a set and $\nu\subseteq H^2$ is a  \emph{quasiordering}, that is a reflexive, transitive relation, on $H$. Quasiordered sets are also called \emph{preordered sets}. 
If $g\in H$ and $\pair xg\in\nu$ for all $x\in H$, then $g$ is a \emph{greatest element} of $H$; \emph{least elements} are defined dually. They are not necessarily unique; if they are, then they are denoted by $1_H$ and $0_H$. Given $H\neq \emptyset$, the set of all quasiorderings on $H$ is denoted by $\Quord H$. It is a complete lattice with respect to set inclusion. Therefore, for $X\subseteq H^2$, there exists a least quasiordering on $H$ that includes $X$; it  is denoted by $\iquos HX$ or  $\quos X$. 
%We write $\quo ab$ rather than $\quos{\set{\pair ab}}$.
%
If $\tuple{H;\nu}$ is a quasiordered set, then $\Theta_\nu=\nu\cap\nu^{-1}$ is 
%known to be 
an equivalence relation, and the definition $\pair{\blokk x{\Theta_\nu}}{ \blokk y{\Theta_\nu}}\in \nu/\Theta_\nu  \iff \pair xy\in\nu$  turns the quotient set $H/\Theta_\nu$ into an ordered set $\tuple{H/\Theta_\nu;\nu/\Theta_\nu}$. 
For an ordered set $H$ and  $x,y\in H$, $\pair xy$ is called an \emph{ordered pair} of $H$ if $x\leq y$. This notation fits   to previous work 
on (principal) lattice congruences.   The set of ordered pairs of $H$ is denoted by $\pairs H$. 

We need the concept of strong auxiliary structures from \init{G.\ }Cz\'edli~\cite{czgprincc}; however, we do not need all the details.
%, which take more than a page in \cite{czgprincc}. 
In particular, the reader does not have to know what the axioms (A1),\dots,(A13) are. By a  \emph{strong auxiliary structure} we mean a structure
\begin{equation}
\alg L=\tuple{L;\gamma, H,\nu,\delta ,\epsilon,\zigzags}\label{auxstr}
\end{equation}
such that the axioms (A1),\dots,(A13) from  \cite{czgprincc} hold. What we only have to know is the following. If $\alg L$ in \eqref{auxstr} is a strong auxiliary structure, then 
$L$ is a bounded lattice, $\tuple{H;\gamma}$ is a quasiordered set, $\gamma\colon \pairs L\to H$ is a map (called quasi-coloring), $\delta$ and $\epsilon$ are maps from $H$ to $L\setminus\set{0_L, 1_L}$, $\delta(p)\preceq \epsilon(p)$ for all $p\in H$, and $\zigzags$ is a set of certain 9-tuples of $L$. 
The following statement follows trivially from (A1), (A4) and the (short) proof of  Lemma~2.1 in \cite{czgprincc}.

\begin{lemma}\label{xilemma}
 If $\alg L$ in \eqref{auxstr} is a strong auxiliary structure, then the map
\[\mu\colon \tuple{H/\Theta_\nu;\nu/\Theta_\nu}\to\tuple{\princ L;\subseteq}\text{, defined by } \blokk p{\Theta_\nu} \mapsto\cgi L {\delta(p)}{\epsilon(p)},
\]
is an order isomorphism.
\end{lemma}

This lemma shows the importance of strong auxiliary structures, and it explains why we are going to construct  quasiordered sets in our proofs.

In the rest of the paper, $F\colon \lcat S\to \catb$ is a functor as in Theorem~\ref{thmmain}. To ease the notation, we will write $\tuple{P_i;\nu_i}$ (or  $\tuple{P_i;\leq_i}$) and 
$\psi_{ij}$ instead of $F(i)$ and  $F(i\leq j)$, 
respectively. We can assume that
$|P_i|\geq 2$ for all $i\in S$, because otherwise (using  that $\xi_{0j}$ and $\xi_{j1}$ are $\set{0,1}$-preserving for all $j\in S$) we obtain that all the $P_i$, $i\in S$, are 1-element, and the theorem trivially holds. The least element and the greatest element of $P_i$ are denoted by $0_i$ and $1_i$, respectively. We can also assume that 
\begin{equation}
\text{for $i\neq j\in S$, $0_i=0_j$, $1_i=1_j$, and $|P_i\cap P_j|=2$.}
\label{mdfuTw}
\end{equation}
In the opposite case, we take two new elements, $0$ and $1$, outside $\bigcup\set{P_i:i\in S}$. Let
$P_i'=(P_i\setminus\set{0_i,1_i})\cup\set{0,1}$. We define an ordering $\leq_i'$ on $P_i'$ such that the map
\[
\alpha_i\colon \tuple{P_i;\nu_i}\to \tuple{P_i';\nu_i'}, 
\text{ defined by }x\mapsto
\begin{cases}
x,&\text{if }x\in P_i\setminus\set{0_i,1_i},\cr
0&\text{if }x=0_i,\cr
1&\text{if }x=1_i.\cr
\end{cases}
\]
is an isomorphism. We let $\psi'_{ij}=  \alpha_j \circ \psi_{ij}\circ\alpha_i^{-1}$. Let $F'\colon \lcat S\to \catb$ be defined by $F'(i)= \tuple{P_i';\leq_i'}$ and $F(i\leq j)=\psi'_{ij}$. This functor is naturally isomorphic to $F$, because $\alpha\colon F\to F'$ is a natural isomorphism. Therefore, if \eqref{mdfuTw} fails, then we can work with $F'$ instead of $F$. This justifies the assumption \eqref{mdfuTw}.

For $j\in S$, let $R_j=\bigcup\set{P_i: i\leq j}$. 
Observe that  $\nu_i\subseteq R_j^2$,  $\psi_{ij}\subseteq R_j^2$ and $\psi_{ij}^{-1}=\set{\pair xy: x=\psi_{ij}(y)}\subseteq R_j^2$ for all $i\leq j$. Hence, we can let   
\[ \hatnu_j = \biquos {R_j}{\bigcup\set{\nu_i: i\leq j, \,\, i \in S}\cup \bigcup\set{\psi_{ij} : i\leq j}\cup \bigcup\set{\psi_{ij} ^{-1}: i\leq j}}\text.
\]
Let $\hatTheta_j= \hatnu_j\cap \hatnu_j^{-1}$.

\begin{lemma} \label{kappalemma} For $j\in S$, the map $\kappa_j\colon \tuple{ R_j/\hatTheta_j;\hatnu_j/ \hatTheta_j}\to \tuple{P_j;\nu_j}$, defined by
\begin{equation} \kappa_j(\blokk x{\hatTheta_j})=\begin{cases}
x,&\text{if } x\in P_j,\cr
\psi_{ij}(x),&\text{if } x\in P_i,
\end{cases} \label{dkdgjhT}
\end{equation}
is an order isomorphism. $($Note that the first line of \eqref{dkdgjhT} can be omitted, since $\psi_{jj}$ is the identity map.$)$
\end{lemma}

\begin{proof}  
Consider the auxiliary map $\kappa_j'\colon \tuple{ R_j;\hatnu_j}\to \tuple{P_j;\nu_j}$, defined by $\psi_{ij}(x)$ for $x\in P_i$.
First, we show that $\kappa_j'$ is monotone, that is, for all $x,y\in R_j$, 
\begin{equation} 
\text{if}\quad \pair xy\in\hatnu_j,\quad\text{then}\quad \pair{\kappa'_j( x )}{\kappa'_j( y)}\in\nu_j\text.
\label{sdhgH}
\end{equation}
Assume that $\pair xy\in\hatnu_j$. By the definition of $\hatnu_j$, there is an $n\in\mathbb N_0$ and there are $z_0,\dots, z_n\in R$ and  $k_1,\dots, k_n\in S$ such that     $z_0=x$, $z_n=y$, $k_1\leq j$, \dots, $k_n\leq j$, and $\pair{z_{i-1}}{z_i}\in \nu_{k_i}\cup\psi_{k_ij}\cup\psi^{-1}_{k_ij}$ for $i\in\set{1,\dots,n}$. 
If $\pair{z_{i-1}}{z_i}\in \nu_{k_i}$, then $\pair{\kappa'_j(z_{i-1})}{\kappa'_j(z_i)} = \pair{\psi_{k_ij}(z_{i-1})}{\psi_{k_ij}(z_i)}\in \nu_j$, since $\psi_{k_ij}$ is monotone.   If $\pair{z_{i-1}}{z_i}\in \psi_{k_ij}$, that is $\psi_{k_ij}(z_{i-1})=z_i$, then $\pair{\kappa'_j(z_{i-1})}{\kappa'_j(z_i)} = 
\pair{z_i}{z_i}\in \nu_j$ by reflexivity.  
Similarly, if 
$\pair{z_{i-1}}{z_i}\in \psi_{k_ij}^{-1}$, that is $\psi_{k_ij}(z_{i})=z_{i-1}$, then $\pair{\kappa'_j(z_{i-1})}{\kappa'_j(z_i)} = \pair{z_{i-1}}{z_{i-1}}\in \nu_j$. 
Thus, $\pair{\kappa'_j(z_{i-1})}{\kappa'_j(z_i)} \in\nu_j$ holds for all $i\in\set{1,\dots,n}$, and 
 $\pair{\kappa'_j(x)}{\kappa'_j(y)}\in\nu_2$ follows 
by the transitivity of $\nu_j$. This proves \eqref{sdhgH}.

Next, if $\blokk x{\hatTheta_j}=\blokk y{\hatTheta_j}$, then $\pair xy,\pair yx\in\hatnu_j$. So, \eqref{sdhgH}  and  the antisymmetry of $\nu_j$  yield that $\kappa'_j(x)=\kappa'_j(y)$. Hence, $\kappa_j$ is a map. Note the rule $\kappa_j(\blokk x{\hatTheta_j})=\kappa'_j(x)$. This, together with \eqref{sdhgH}, implies that $\kappa_j$ is monotone.  Since     $\kappa'_j$ is surjective,  so is $\kappa_j$. 
Hence, to complete the proof, it suffices to show that
\begin{equation}
\text{if}\quad \pair{\kappa_j(\blokk x{\hatTheta_j})}{\kappa_j(\blokk y{\hatTheta_j})}\in\nu_j,\quad\text{then}\quad \pair{\blokk x{\hatTheta_j}} {\blokk y{\hatTheta_j}}\in\hatnu_j/\hatTheta_j\text.
\label{slidGuZ}
\end{equation}
Assume that $\pair{\kappa_j(\blokk x{\hatTheta_j})}{\kappa_j(\blokk y{\hatTheta_j})}\in\nu_j $. This means that $\pair{\kappa'_j(x)}{\kappa'_j(y)}\in \nu_j$, and we have to show that $\pair xy\in \hatnu_j$. By the definition of $R_j$, there are $i,k\in S$ with $i\leq j$ and $k\leq j$ such that $x\in P_i$ and $y\in P_k$. We obtain $\pair{x}{\kappa'_j(x)}= \pair{x}{\psi_{ij}(x)}\in\psi_{ij}\subseteq \hatnu_j$, we already mentioned  $\pair{\kappa'_j(x)}{\kappa'_j(y)}\in \nu_j$, and we also have $\pair{\kappa'_j(y)}{y}= \pair{\psi_{kj}(y)}{y}\in\psi_{kj}^{-1}\subseteq \hatnu_j$. Hence $\pair xy\in \hatnu_j$ follows by transitivity. This proves \eqref{slidGuZ} and the lemma.
\end{proof}

\begin{proof}[Proof of Theorem~\ref{thmmain}]
For $i\in S$, let $\nu_i^-=(\set{0_i}\times P_i)\cup (P_i\times \set{1_i})$; note that 
$\tuple{P_i;\nu_i^-}$ is a modular lattice of length 2. 
Let $\alg L_i^-= \tuple{L_i^-;\gamma_i^-, P_i,\nu_i^-,\delta_i^-,\epsilon_i^-,\zigzags_i^-}$ 
denote the strong auxiliary structure defined in Example 2.2 (and Figure  5) of \cite{czgprincc}, with $\tuple{P_i;\nu_i^-}$  playing the role of $\tuple{H;\nu}$.  
Similarly, for $j\in S$, let  
$\hatnu_j^-=(\set{0_{R_j}}\times R_j)\cup (R_j\times \set{1_{R_j}})$, and let  $\hat{\alg L}_j^-= \tuple{\hat L_j^-;\hat\gamma_j^-, R_j,\hatnu_j^-,\hat\delta_j^-,\hat\epsilon_j^-,\hat{\zigzags}_j^-}$ denote the strong auxiliary structure defined in Example 2.2 (and Figure  5) of \cite{czgprincc}, with $\tuple{R_j;\hatnu_j^-}$ playing the role of $\tuple{H;\nu}$. Again,  $\tuple{R_j;\hatnu_j^-}$ is a modular lattice of length 2.
For $i\leq j\leq k \in S$,   the construction  in \cite{czgprincc} yields trivially  that $L_i^-$ is a $\set{0,1}$-sublattice of $\hat L_j^-$,   $\hat L_j^-$  is a $\set{0,1}$-sublattice of $\hat L_k^-$, 
$\delta_i^-$ is the restriction $\restrict{\hat\delta_j^-}{P_i}$  of $\hat\delta_j^-$ to $P_i$,   
 $\hat\delta_j^-  =  \restrict{\hat\delta_k^-}{R_j} $ and,  analogously,
$\epsilon_i^-  = \restrict{\hat\epsilon_j^-}{P_i}$ and  
 $\hat\epsilon_j^-  =  \restrict{\hat\epsilon_k^-}{R_j} $. 
Since  $\nu_i ^-\subseteq \nu_i$, 
we can apply \cite[Lemma 5.3]{czgprincc} so that $\tuple{P_i,\nu_i ^-,P_i,\nu_i}$ plays the role of $\tuple{H,\nu,\vesz H,\vesz \nu}$.  
In this way, we obtain a strong auxiliary structure 
$\alg L_i= \tuple{L_i;\gamma_i, P_i,\nu_i,\delta_i,\epsilon_i,\zigzags_i}$.  
Note that $L_i ^-$ is a sublattice of $L_i$, $\delta_i=\delta_i ^-$, and $\epsilon_i=\epsilon_i ^-$. Let $\hatnu_{ij} =\iquos {R_j}{\nu_i}=\nu_i\cup \hatnu_j ^-$.  
Giving the role of $\tuple{H,\nu,\vesz H,\vesz \nu}$ to $\tuple{R_j,\hatnu_j ^-,R_j,\hatnu_{ij}}$, \cite[Lemma 5.3]{czgprincc} yields a strong auxiliary structure
$\hat{\alg L}_{ij} = \tuple{\hat L_{ij};\hat\gamma_{ij},R_j,\hatnu_{ij},\hat\delta_{ij},\hat\epsilon_{ij},\hat\zigzags_{ij}}$. 
It is clear from the construction, which is described in \cite{czgprincc}, that $\hat L_j^-$ is a $\set{0,1}$-sublattice of $\hat L_{ij}$, $\hat\delta_{ij}=\hat\delta_j^- $, and  $\hat\epsilon_{ij}=\hat\epsilon_j^- $.

Finally, using  \cite[Lemma 5.3]{czgprincc} with $\tuple{R_j,\hatnu_{ij},R_j,\hatnu_j}$ in place of $\tuple{H,\nu,\vesz H,\vesz \nu}$, we obtain a strong auxiliary structure 
$\hat{\alg L}_j= \tuple{\hat L_j;\hat\gamma_j, R_j,\hatnu_j,\hat\delta_j, \hat\epsilon_j,\hat{\zigzags}_j}$. Again, the construction yields that $\hat L_{ij}$ is a $\set{0,1}$-sublattice of $\hat L_j$, $\hat\delta_j=\hat\delta_{ij}$, and $\hat\epsilon_j=\hat\epsilon_{ij}$. It also yields that, for $j\leq k$,  
\begin{equation}
\text{
 $\hat L_j$  is a $\set{0,1}$-sublattice of $\hat L_k$, $\hat\delta_j=
\restrict {\hat\delta_k}{R_j}$, and $\hat\epsilon_j=
\restrict {\hat\epsilon_k}{R_j}$.}
\label{dlkdzNtN}
\end{equation}

Now, let $L=\hat L_1$. Here $1$ is the top element of $S$, whence all the $\hat L_j$, $j\in S$, are objects of  the category $\subecat L$. For $j\in S$, we define $E(j)$ by $E(j)=\hat L_j$; we obtain from \eqref{dlkdzNtN} that $E$ is a functor.
 For $j\in S$,  we obtain from Lemma~\ref {xilemma} that 
\[\mu_j\colon \tuple{R_j/\hatTheta_j;\hatnu_j/\hatTheta_j}\to\tuple{\princ{ \hat L_j};\subseteq}\text{, defined by } \blokk p{\hatTheta_j} \mapsto\cgi {\hat L_j} {\hat\delta_j(p)}{\hat\epsilon_j(p)},
\]
is an order isomorphism. So is $\kappa_j$ by Lemma~\ref{kappalemma}.  Hence the composite map 
\[  \xi_j=  \mu_j \circ \kappa_j^{-1},\quad\text{from }
F(j)=\tuple{P_j;\nu_j}\,\, \text{ to }\,\,  (\fprinc\circ E)(j) = \tuple{\princ{ \hat L_j};\subseteq},
\]
is also an order isomorphism. To show that $\xi$, defined by $\xi(j)=\xi_j$, is a natural isomorphism from $F$ to $\fprinc\circ E$, we have to prove that, for $j\leq k\in S$, the diagram
\begin{equation}
\begin{CD}
\tuple{P_j;\nu_j}   @>{\dd\psi_{jk}}>> \tuple{P_k;\nu_k} \\
@V{\dd{\xi_j}}VV   @V{\dd{\xi_k }}VV   \\
\tuple{\princ{ \hat L_j};\subseteq} @>{\dd{\zetai {\hat L_j}{\hat L_k}}}>>  \tuple{\princ{ \hat L_k};\subseteq}
\end{CD}\\ \label{vkdhHnf}
\end{equation}
commutes. To do so, consider an arbitrary element $p\in P_j$.  Using $\hat\delta_j=\restrict{\hat\delta_k}{R_j}$ and $\hat\epsilon_j=\restrict{\hat\epsilon_k}{R_j}$,  we have that 
\begin{align}
\zetai {\hat L_j}{\hat L_k}(\xi_j(p))=\zetai {\hat L_j}{\hat L_k}\bigl( \cgi{\hat L_j}{\hat\delta_j(p)}{\hat\epsilon_j(p)}\bigr) =  \cgi{\hat L_k}{\hat\delta_k(p)}{\hat\epsilon_k(p)}\bigr)\text. 
\label{cMtsa}
\end{align}
By \eqref{dkdgjhT}  and $p\in P_j$, we have $\kappa_k(\blokk p{\hatTheta_k}) =  \psi_{jk}(p)$. This gives 
$\kappa_k^{-1}(\psi_{jk}(p) ) = \blokk p{\hatTheta_k}$. Hence,
\begin{align}
\xi_k(\psi_{jk}(p)) =  \mu_k \bigl(  \kappa_k^{-1}(\psi_{jk}(p) )    \bigr ) =  \mu_k \bigl(  \blokk p{\hatTheta_k}  \bigr )  =\cgi{\hat L_k}{\hat\delta_k(p)}{\hat\epsilon_k(p)}\bigr)\text. \label{cMtsb}
\end{align}
Thus, we conclude from \eqref{cMtsa} and \eqref{cMtsb} that \eqref{vkdhHnf} is a commutative diagram. This   proves the theorem.
\end{proof}

\begin{conclude}  Gillibert and  Wehrung~\cite{gillibertwehrung} gives a general theory how to turn a representation theorem for a single homomorphism to a representation theorem of a functor, that is, a representation theorem for many homomorphisms. However, the arrows (morphisms) in \cite{gillibertwehrung}  seem to go in the opposite direction, and we do not see how \cite{gillibertwehrung} could offer an easier approach to Theorem~\ref{thmmain}
\end{conclude}

%itt tartok

\end{document}